\documentclass[a4paper,11pt]{article}

\usepackage{amsfonts}
\usepackage{amsopn}
\usepackage{amsmath}
\usepackage{amsthm}
\usepackage{latexsym}
\usepackage{amssymb}
\usepackage{color}

\newtheorem{theorem}{Theorem}

\newtheorem{remark}[theorem]{Remark}

\newtheorem{proposition}[theorem]{Proposition}

\newcommand{\p}{\partial}
\newcommand{\eps}{\varepsilon}

\newcommand{\dd}{\hspace{1pt}{\rm d}\hspace{0.5pt}}

\newcommand{\ee}{{\rm e}\hspace{1pt}}

\newcommand{\ii}{{\rm i}\hspace{1pt}}

\newcommand{\R}{\mathbb R}
\newcommand{\N}{\mathbb N}
\newcommand{\C}{\mathbb C}

\newcommand{\cB}{{\mathcal B}}
\newcommand{\cC}{{\mathcal C}}
\newcommand{\cD}{{\mathcal D}}
\newcommand{\cE}{{\mathcal E}}
\newcommand{\cF}{{\mathcal F}}

\newcommand{\cM}{{\mathcal M}}

\newcommand{\cS}{{\mathcal S}}

\newcommand{\bone}{\mathbf{1}}

\newcommand{\Cinf}{\ensuremath{{\mathcal C}^\infty}}

\newcommand{\WF}{\mathrm{WF}}
\newcommand{\singsupp}{\mathop{\mathrm{sing supp}}}
\newcommand{\supp}{\mathop{\mathrm{supp}}}
\renewcommand{\Re}{\mathop{\mathrm{Re}}}

\date{}

\title{Solutions to semilinear wave equations of very low regularity}

\author{Heiko Gimperlein, Michael Oberguggenberger\thanks{Universit\"{a}t Innsbruck, Engineering Mathematics,
Technikerstra\ss e 13, 6020 Innsbruck,
Austria, (heiko.gimperlein@uibk.ac.at, michael.oberguggenberger@uibk.ac.at)}
}

\begin{document}

\providecommand{\keywords}[1]{{\textit{Key words:}} #1}

\maketitle

\begin{abstract}
\noindent This paper finds solutions to semilinear wave equations with strongly anomalous propagation of singularities. For very low Sobolev regularity we obtain solutions whose singular support propagates along any ray inside or outside the light cone. In one dimension these solutions exist for any Sobolev exponent $s<\frac{1}{2}$ in space, while classical results show that the singular support of solutions with higher regularity is contained in the light cone. The spatial Fourier transform of these anomalous solutions is supported in a half-line. We obtain wellposedness results in such function spaces when the problem is ill-posed for Sobolev data without the support condition and, in some cases, obtain wellposedness below $L^2(\R)$. The results are based on new multiplier theorems for Sobolev spaces satisfying the support condition. Extensions to higher space dimensions are given.

\end{abstract}


\section{Introduction}
\label{sec:intronew}

This paper observes new phenomena for the wellposedness and propagation of singularities for semilinear wave equations with initial data of very low Sobolev-regularity. We address the problem
\begin{equation}\label{eq:NLW}
  \p_t^2 u - \Delta u = \pm u^p,\quad u(x,0) = u_0(x), \ \p_t u(x,0) = u_1(x),
\end{equation}
in space dimension $n$, where  $p\geq 2$ is assumed to be a positive integer. The specialization to this case is needed for two reasons. First, we wish to study propagation of singularities from the initial data, and hence we need a smooth nonlinearity in order to avoid the occurrence of additional singularities. Second, we will use the H\"ormander product of distributions, so only integer powers are amenable.

It is a general principle in linear wave propagation that sharp wave crests (singularities) propagate along light cones, or more precisely, along the bicharacteristics of the linear wave operator $\p_t^2 u - \Delta u$. For semilinear wave equations, singularities of sufficiently smooth solutions are known to propagate along unions of bicharacteristics, where in space dimension $n>1$  new -- but weaker -- singularities
may arise at points of intersection of incoming wave crests. In dimension $n=1$, a fundamental result for problem \eqref{eq:NLW} assures that the singularities of any  solution in $L^\infty_{\rm loc}(\R^2)$ propagate only along light cones \cite{RauchReed:1980,Reed:1978}.

In this article we obtain solutions
in $\cC([-T,T]:H^s_{\rm loc}(\R))$, for any $s < \frac12$, whose singular support lies inside or outside the light cone. As $\cC([-T,T]:H^s_{\rm loc}(\R)) \subset L^\infty_{\rm loc}(\R^2)$ when $s>\frac{1}{2}$, they establish a sharp threshold $s=\frac{1}{2}$ for the Sobolev exponent between solutions with expected, respectively anomalous, propagation of singularities. Note that the maximal  regularity of solutions with unexpected properties has attracted significant recent interest for equations from continuum mechanics and geometry \cite{DeLellis:1999}.

For the anomalous solutions $u$ presented here, at fixed time $t$ the Fourier transform $\widehat{u}$ with respect to $x$  is supported in a half-line.
Motivated by this fact, we also extend the range of wellposedness for problem \eqref{eq:NLW} to data and solutions $u$ with this property. 
The paper addresses two different, yet related issues: (a) anomalous propagation of singularities for solutions to \eqref{eq:NLW} of low Sobolev regularity in one and higher space dimension and (b) wellposedness for a certain class of initial data of low Sobolev regularity in one space dimension.
The main results of this article can be summarized as follows:\\

\noindent \textbf{Theorem.} Let $n=1$ and $H^s_\Gamma(\R) = \{f\in H^s(\R): \supp \widehat{f}\subset [0,\infty)\}$.

(a) (Anomalous propagation) For any $c\neq \pm 1$ there exist solutions to \eqref{eq:NLW}  with singular support along the line $\{x + ct= 0, t\in\R\}$, i.e., along  any ray off the light cone.
More precisely, for any $c\neq \pm 1$ and any $s < \frac12$ there are $H_{\rm loc}^s$-solutions with this property.

(b) (Low regularity wellposedness) Problem \eqref{eq:NLW} is wellposed in $H^s_\Gamma(\R)$ for $p=2$ and $s > -\frac12$ and for $p\geq 3$ and $s > \frac{1}{2}-\frac{1}{2p-4}$.\\

Solutions to \eqref{eq:NLW} are understood as follows: the nonlinearity is defined by H\"{o}rmander's wave front set criterion, and the equation is satisfied in the sense of distributions.
Part (a) combines Proposition \ref{prop:NLW1D} with the discussion in Section \ref{sec:stationary}. The construction builds on recent examples found by one of the authors \cite{MO:2021};
for any $s<\frac12$ there is such a solution in $\cC([-T,T]:H^s_{\rm loc}(\R))\cap\cC^1([-T,T]:H^{s-1}_{\rm loc}(\R))$, if $p$ is large enough. Part (b) is the content of Theorem \ref{thm:1DNLW}. Note that it improves the wellposedness results for data in $H^s(\R)$ without the support condition, where problem \eqref{eq:NLW} is wellposed  if  $s>\frac12 - \frac1{p}$ and illposed if $s<\frac12 - \frac1{p}$.
The solutions constructed in (a) do not fall into the known $H^s$-wellposedness regimes.

Section \ref{sec:ndim} addresses the extension of this Theorem to higher space dimensions $n>1$. We give examples of anomalous solutions to \eqref{eq:NLW} which belong to
$\cC([-T,T]:H^s_{\rm loc}(\R^n))\cap \cC^1([-T,T]:H^{s-1}_{\rm loc}(\R^n))$ for $s < \frac{n}2$. Here the nonlinearity is even defined classically as the $p$-th power of an $L^p_{\rm loc}$-function. A similar extension of the wellposedness theory remains open.

The remainder of this introduction is devoted to a literature review of propagation of singularities and of critical Sobolev exponents.

The investigation of propagation of singularities in semilinear hyperbolic equations and systems started with the discovery of Jeffrey Rauch and Michael Reed \cite{RauchReed:1980,RauchReed:1981} that -- unlike in the linear case -- singularities may arise that cannot be traced back via bicharacteristics to singularities in the initial data, but may be produced at later times by the interaction of singularity bearing bicharacteristics. For a survey of the huge number of results up to around 1990 we refer to the monograph \cite{Beals:1989}. Rauch and Reed coined the term \emph{anomalous singularities} for this phenomenon. However, these ``anomalous singularities'' still propagated along characteristics/bicharacteristics, as opposed to the noncharacteristic singularities in the present paper, which are even more anomalous.

There is one exception, namely the wave equation $\p_t^2 u - \p_x^2 u = f(u)$ with smooth nonlinearity $f(\cdot)$ in one space dimension (actually any $(2\times 2)$-first order system in $n=1$) where the propagation is as in the linear case. This is due to the fact that there are only two characteristic directions, thereby avoiding nonlinear interaction at later times. Here the results of \cite{RauchReed:1980,Reed:1978} say that for distributional solutions to the semilinear wave equation which belong to $L^\infty_{\rm loc}(\R^2)$ no anomalous singularities arise. (This applies, in particular, to solutions which belong to $\cC([-T,T]:H^s(\R))$ with $s > \frac12$.) For example, if the singular support of the initial data is $\{x=0\}$ then the solution is smooth except possibly along the light cone $\{|x|=|t|\}$.

In higher space dimensions, the first and prototypical result is due to Rauch \cite{Rauch:1979}. It says the following: Suppose that $u$ is a distributional solution to \eqref{eq:NLW} (even with a polynomial nonlinearity) which belongs to $H^s_{\rm loc}(\R^n\times\R)$ with $s>(n+1)/2$ and let the initial data belong to $\Cinf(\R^n\setminus\{0\})$. Then $u$ is $\Cinf$ on $\{|x| > |t|\}$, and it belongs to $H^{s+1+\sigma}_{\rm loc}(\R^n\times\R)$ on $\{|x| < |t|\}$ for all $\sigma < s - (n+1)/2$. It is also known that the singular support of the solution may contain the solid cone $\{|x| \leq |t|\}$, see \cite{Beals:1983}, where an example is given with Sobolev regularity just above $3s -n +2$ in  $\{|x| < |t|\}$.

These results date back to a time when the investigation of critical exponents had not yet been picked up. Accordingly, the usual setting was in $H^s_{\rm loc}(\R^n\times\R)$ with $s>(n+1)/2$, in which case $H^s_{\rm loc}$ is an algebra. The methods were commonly based on a microlocal analysis of the nonlinear action \cite{BealsReed:1982,RauchReed:1982}, as well as on paradifferential calculus \cite{Bony:1981}. Few papers addressed propagation of local regularity in lower Sobolev regularity, as the paper \cite{GerardRauch:1987} which went as low as $s>0$ (but still requiring $L^\infty_{\rm loc}$); see also the early counterexamples of anomalous bicharacteristic behavior in low regularity in the second part of \cite{Rauch:1981}.

In the meantime the wellposedness of problem \eqref{eq:NLW} for data of low regularity has been clarified. Recall that problem \eqref{eq:NLW} is locally wellposed in $H^s$ if, for every $u_0\in H^s(\R^n)$, $u_1\in H^{s-1}(\R^n)$, there is $T>0$ and a unique distributional solution $u$ belonging to $\cC([-T,T]:H^s(\R^n))\cap\cC^1([-T,T]:H^{s-1}(\R^n))$. Further, $u$ is required to belong to a space on which the $p$-th power is welldefined (usually $L^p_{\rm loc}(\R^{n+1}$)), and the map $(u_0,u_1)\to u$ should be continuous. The $p$-th power here may  also be understood as a Fourier product, see Section \ref{Sec:multiplication}.

As summarized in \cite{Christ:2003,Forlano:2020}, the critical regularity for local $H^s$-wellposedness of problem \eqref{eq:NLW} is
\[
   s_{\rm crit} = \max\left(\frac{n}2 - \frac{2}{p-1},\frac{n+1}4 - \frac{1}{p-1},0\right).
\]
Essentially, wellposedness has been established for $s\geq s_{\rm crit}$, possibly with additional constraints in certain ranges of $p$ and $n$, while illposedness has been proven for $s < s_{\rm crit}$, again with certain gaps in the ranges. Relevant literature is \cite{Kapitanski:1994,Keel:1998,Lindblad:1993,Lindblad:1996,LindbladSogge:1995,Tao:1999}, as well as recent directions for the probabilistic wellposedness \cite{Burq:2008,Oh:2016,Oh:2021,Oh:2022,Tzvetkov:2019}. For more details, the reader is referred to the summaries in \cite{Christ:2003,Forlano:2020}. The case $n=1$ deserves special attention. The critical exponent is
\begin{equation}\label{eq:sob}
     s_{\rm sob} = \max\left(\frac12 - \frac1{p}, 0\right).
\end{equation}
The stronger bound is needed in order to have $H^s(\R) \subset L^p(\R)$. It was shown in \cite{Christ:2003} that problem \eqref{eq:NLW} is $H^s$-illposed for $s <  s_{\rm sob}$. In addition, it was shown there that norm inflation takes place for
$\frac12 - \frac1{p-1} < s <  s_{\rm sob}$ and for $s\leq -\frac12$. Further, the authors also showed that the solution map is discontinuous at $(0,0)$ for $s\leq \frac12 - \frac1{p-1}$. These results were complemented by \cite{Forlano:2020} which proved norm inflation also in the range $s < 0$. It is also noted in \cite{Christ:2003} that problem \eqref{eq:NLW} is locally $H^s$-wellposed when $n=1$ and $s \geq s_{\rm sob}$.

\section{Notation}
\label{sec:notations}

The notation generally follows \cite{Schwartz:1966}. In particular, the Fourier transform is used in the form
\[
   \cF f(\xi) = \widehat{f}(\xi) = \int \ee^{-2\pi\ii x\xi} f(x)\dd x.
\]
As usual, $\langle \xi \rangle = (1 + |\xi|^2)^{1/2}$. For $s\in\R$, we write
\[
   L^2_s(\R^n) = \{h\in\cS'(\R^n):\langle \xi \rangle^s h(\xi) \in L^2(\R^n)\}.
\]
The Sobolev spaces and local Sobolev spaces, respectively, are defined by
\[
   H^s(\R^n) = \{f\in\cS'(\R^n):  \widehat{f} \in L^2_s(\R^n)\}
\]
and
\[
   H^s_{\rm loc}(\R^n) = \{f\in\cS'(\R^n):  \psi f \in H^s(\R^n) \ {\rm for\ all\ } \psi\in\cD(\R^n)\}.
\]

The distribution $(x+\ii 0)^\lambda \in \cS'(\R)$ is defined as
\begin{equation}\label{eq:xplusi0}
    {(x+\ii 0)^\lambda = }\lim_{\eps\to 0}(x^2 + \eps^2)^{\lambda/2}\ee^{\ii \lambda \arg(x+\ii\eps)}.
\end{equation}
It is an entire function of $\lambda\in\C$, see e.g. \cite[Section I.3.6]{GelfandShilov:1964}. Its Fourier transform is given by
\[
\cF\big((x+\ii 0)^\lambda\big)(\xi) = \frac{(2\pi)^{-\lambda}}{\Gamma(-\lambda)}\ee^{\ii\lambda\pi/2}\xi_+^{-\lambda-1}
\]
for $\lambda \neq 0,1,2,3,\ldots$ \cite[Section II.2.3]{GelfandShilov:1964}, where $\xi_+^\mu$ is the pseudofunction as  defined in \cite[Section I.3.2]{GelfandShilov:1964}.

\begin{remark}\label{rem:HsPrpoertiesOfPseudofunctions}
The following properties are easy to show. We assume here that $\lambda < 0$ so that the pseudofunction $\xi_+^{-\lambda - 1}$ is locally integrable.
Then, for $\lambda < 0$, the following equivalences hold:
\begin{itemize}
\item[(a)]  $(x+\ii 0)^\lambda \in H^s_{\rm loc}(\R)\ \Leftrightarrow\ s < \lambda + \frac12$,
\item[(b)]  $(x+\ii 0)^\lambda \in H^s(\R)\ \Leftrightarrow\ \lambda < -\frac12 \ {\rm and\ }s < \lambda + \frac12$.
\end{itemize}
\end{remark}

\section{Multiplication of distributions}
\label{Sec:multiplication}

This section serves to recall the products of distributions which will be used to define the integer powers in the semilinear wave equation \eqref{eq:NLW}.

Let $S,T\in \cS'(\R^n)$. The $\cS'$-convolution of $S$ and $T$ is said to exist, if
\[
   (\varphi\ast S^-)T \in \cD_{L^1}'(\R^n),\quad {\rm for\ all}\quad \varphi\in \cS(\R^n),
\]
where $S^-(x) = S(-x)$. In this case, the convolution is defined by $\langle S\ast T,\varphi\rangle = \langle (\varphi\ast S^-)T, 1 \rangle$, and $S\ast T$ belongs to $\cS'(\R^n)$.

Let $u,v\in \cS'(\R^n)$. If the $\cS'$-convolution of $\cF u$ and $\cF v$ exists, one may define the \emph{Fourier product}
\begin{equation}\label{eq:FourierProduct}
   u\cdot v = \cF^{-1}(\cF u \ast \cF v).
\end{equation}
The definition can be localized as follows. Assume that for every $x\in\R^n$ there is a neighborhood $\Omega_x$ and $\chi_x\in\cD(\R^n)$, $\chi_x\equiv 1$ on $\Omega_x$, such that the $\cS'$-convolution of $\cF(\chi_x u)$ and $\cF(\chi_x v)$ exists. Locally near $x$, the product $u\cdot v$ is defined to be
$\cF^{-1}(\cF(\chi_x u) \ast \cF(\chi_x v))$. Globally, it is defined by a partition of unity argument.

\begin{remark}\label{rem:FourierProduct}
Here are some special cases in which the Fourier product exists.

(a) The existence of the $\cS'$-convolution of $S,T\in \cS'(\R^n)$ is guaranteed if both $S$ and $T$ have their support in a closed, acute and convex cone $\Gamma$ in $\R^n$. Further, $S\ast T$ is also supported in $\Gamma$, and the map $(S,T)\to S\ast T$ is separately continuous in $\cS'(\R^n)$ \cite[I.5.6, I.4.5]{Vladimirov:1981}. Let
\[
   \cS'_\Gamma(\R^n) = \{f\in \cS'(\R^n): \supp \widehat{f}\subset \Gamma\}.
\]
For $u,v\in \cS'_\Gamma(\R^n)$, the product $u\cdot v$ is thus definable by \eqref{eq:FourierProduct} and belongs to $\cS'_\Gamma(\R^n)$. Thus $\cS'_\Gamma(\R^n)$ forms an algebra with respect to multiplication, and the multiplication map is separately continuous.

(b) Let $u, v \in L^2(\R^n)$, $\varphi\in\cS(\R^n)$. Then $(\varphi\ast\widehat{u}^-)\widehat{v} \in L^1(\R^n) \subset \cD_{L^1}'(\R^n)$. A simple calculation shows that the $\cS'$-convolution $\widehat{u} \ast \widehat{v}$ exists and coincides with the ordinary convolution. By the exchange formula for $L^2$-functions, the Fourier product $u\!\cdot\! v$ coincides with the ordinary product of two $L^2$-functions. Using localization as indicated above, the same holds for the product of two $L^2_{\rm loc}$-functions. In particular, for $u,v \in H^s_{\rm loc}(\R^n)$ with $s > n/2$, the Fourier product exists and coincides with the product in the algebra $H^s_{\rm loc}(\R^n)$.

When $1\leq p < 2$, $2 < q \leq \infty$, $\frac1p + \frac1q = 1$, there are examples of $u\in L^p(\R)$, $v\in L^q(\R)$ whose Fourier product does not exist, as shown in the recent paper \cite{Ortner:2023}. However, if both the ordinary product and the Fourier product exist, they necessarily coincide.

(c) The product defined by H\"ormander's wave front set criterion \cite{Hoermander:1971}, requiring that for every
$(x,\xi)\in \R^n\times(\R^{n}\setminus\{0\})$, $(x,\xi)\in\WF(u)$ implies $(x,-\xi)\not\in\WF(v)$, is also a special case.
This can be seen by localizing the arguments establishing case (a), see e.g. \cite[Proposition 6.3]{MO:1992}.
\end{remark}

The remark shows that all products of functions and distributions occurring in this paper can be subsumed under the framework of the Fourier product.  Further details and a discussion of different products of distributions can be found in \cite{MO:1992}.

\begin{remark}\label{rem:xplusi0}
The distributions $(x+\ii 0)^\lambda$ in \eqref{eq:xplusi0} belong to $\cS'_\Gamma(\R)$ with $\Gamma = [0,\infty)$. Therefore, all integer powers exist both in the sense of the Fourier product and using H\"{o}rmander's wave front set criterion. Their Sobolev regularity is summarized in
Remark\;\ref{rem:HsPrpoertiesOfPseudofunctions}.
\end{remark}

\section{Anomalous propagation of singularities to 1D-semilinear wave equations}
\label{sec:stationary}

In this section, we consider the propagation of singularities for the semilinear wave equation \eqref{eq:NLW} in one dimension.
The following proposition exhibits explicit solutions with stationary singular support. They are used below to construct solutions whose singular support propagates in arbitrary noncharacteristic directions.

\begin{proposition}\label{prop:NLW1D}
For every $s<\frac12$ there are $\lambda<0$ and $p\in\N$ such that

(a) the distribution $u_0(x) = (x+\ii 0)^\lambda$ belongs to $H^s_{\rm loc}(\R)$, $\singsupp u_0 = \{0\}$, and

(b) $u(x,t)\equiv u_0(x)$ is a distributional solution to the semilinear wave equation
\begin{equation}\label{eq:NLW1D}
  \p_t^2 u - \p_x^2 u = -\lambda(\lambda -1)u^p,\quad u(x,0) = u_0(x), \ \p_t u(x,0) = 0
\end{equation}
where the nonlinear term is understood in the sense of the Fourier product. Its singular support is the noncharacteristic line $\{(0,t):t\in\R\}$.

Similarly, for any $1\neq p \in \mathbb{N}$ there are $\lambda$ and $s<\frac12$ such that $u(x,t)\equiv u_0(x)$ is a distributional solution of \eqref{eq:NLW}.
\end{proposition}
Multiplying the solution $u$ of \eqref{eq:NLW1D} by a constant, we obtain a corresponding solution of \eqref{eq:NLW}.

The special solutions exhibited here are self-similar solutions to the semilinear wave equation. However, they do not belong to the classes of functions considered e.g.~in
\cite{Bizon:2007,Kato:2007,Pecher:2000a,Pecher:2000,Ribaud:2002}.

\begin{proof}[Proof of Proposition \ref{prop:NLW1D}]
The function $\C\to\cS'(\R)$, $\lambda \to (x+\ii 0)^\lambda$ is analytic and it is well-known that $\frac{\dd^2}{\dd x^2}(x+\ii 0)^\lambda = \lambda(\lambda -1)(x+\ii 0)^{\lambda-2}$. The support of its Fourier transform is $[0,\infty)$, so all integer powers make sense by means of the Fourier product. Further,
\[
   (x+\ii 0)^{\lambda-2} = (x+\ii 0)^{\lambda p}
\]
provided $\lambda = \frac2{1-p}$. Noting that $\lambda < 0$, Remark \ref{rem:HsPrpoertiesOfPseudofunctions} shows that $(x+\ii 0)^{\lambda}$ belongs to $H^s_{\rm loc}(\R)$ iff $s < \lambda + \frac12 = \frac2{1-p} + \frac12$. Let $p\to\infty$ to produce the desired $s$.

Finally, if $1\neq p \in \mathbb{N}$, setting $\lambda = \frac2{1-p}$ produces a solution in $H^s_{\rm loc}(\R)$ for $s < \frac12 - \frac{2}{p-1}$.

\end{proof}

\begin{remark} (a) Anomalous propagation of singularities. When $s>\frac12$, equation \eqref{eq:NLW1D} with initial data in $H^s(\R)\times H^{s-1}(\R)$ would have a unique solution which belongs to $\cC^0([-T,T]:H^s(\R^n)) \subset L^\infty(\R\times[-T,T])$, so the anomalous singular support of the solution from Proposition \ref{prop:NLW1D} would be ruled out by the results in \cite{RauchReed:1980,RauchReed:1982}.

(b) Critical exponents. By Remark \ref{rem:HsPrpoertiesOfPseudofunctions}, $(x+\ii 0)^{\lambda}$ with $\lambda = \frac2{1-p}$ actually belongs to $H^s(\R)$ for $p = 2$ and $p=3$,
where $s < \frac2{1-p} + \frac12$. However, this is outside the range of wellposedness given by Theorem \ref{thm:1DNLW} and, furthermore, below $s_{\rm sob}$ in \eqref{eq:sob}.
\end{remark}

Using certain Lorentz transformations, it is possible to transform the stationary solutions $u_0(x) = (x+\ii 0)^\lambda$ from Proposition \ref{prop:NLW1D}
to time dependent solutions with singular support on noncharacteristic rays. Starting with the case $n=1$, the transformation
\[
  \begin{bmatrix}x\\ t\end{bmatrix} \to L \begin{bmatrix}x\\ t\end{bmatrix}, \quad L = \begin{bmatrix}\cosh\theta & \sinh\theta\\ \sinh\theta & \cosh\theta\end{bmatrix}
\]
keeps the quadratic form $x^2 - t^2$ invariant, while the transformation
\[
  \begin{bmatrix}x\\ t\end{bmatrix} \to L' \begin{bmatrix}x\\ t\end{bmatrix}, \quad L' = \begin{bmatrix}\sinh\theta & \cosh\theta\\ \cosh\theta & \sinh\theta\end{bmatrix}
\]
keeps the quadratic form $t^2 - x^2$ invariant. Therefore, if $u(x,t)$ solves
\begin{equation}\label{eq:Lorentz}
  \p_t^2 u - \p_x^2 u = f(u),
\end{equation}
then $v(x,t) = u \circ L(x,t)$ and $w(x,t) = u \circ L'(x,t)$ solve
\begin{equation*}
  \p_t^2 v - \p_x^2 v = -f(v),\qquad \p_t^2 w - \p_x^2 w = +f(w),
\end{equation*}
respectively. In particular, if $u(x,t) \equiv u_0(x)$ is a stationary solution to \eqref{eq:Lorentz} (with $\p_t u(x,0) = 0$), then
$v(x,t) = u_0(x\cosh\theta + t\sinh\theta)$ solves the same equation with a sign change and with initial data
\[
v(x,0) = u_0(x\cosh\theta),\quad \p_t v(x,0) = \sinh\theta\, u_0'(x\cosh\theta)
\]
while $w(x,t) = u_0(x\sinh\theta + t\cosh\theta)$ solves \eqref{eq:Lorentz} with initial data
\[
w(x,0) = u_0(x\sinh\theta),\quad \p_t v(x,0) = \cosh\theta\, u_0'(x\sinh\theta).
\]
Suppose now that $u_0(x)$ has singular support equal to $x = 0$. The reparametrization
\[
    \cosh\theta = \frac{1}{\sqrt{1-c^2}},\quad \sinh\theta = \frac{c}{\sqrt{1-c^2}}
\]
with $|c| < 1$ leads to
\[
   v(x,t) = u_0\Big(\frac{x+ct}{\sqrt{1-c^2}}\Big)
\]
which has its singular support along the line $\{x + ct= 0, t\in\R\}$, that is, inside the light cone, while the singular support of the initial data is still $\{x=0\}$.
Similarly, the reparametrization
\[
    \cosh\theta = \frac{c}{\sqrt{c^2-1}},\quad \sinh\theta = \pm\frac{1}{\sqrt{c^2-1}}
\]
with $c > 1$ leads to
\[
   w(x,t) = u_0\Big(\frac{\pm x + ct}{\sqrt{c^2-1}} \Big)
\]
which has its singular support along the line $\{x \pm ct= 0, t\in\R\}$, that is, outside the light cone. In conclusion, the stationary solutions from Proposition \ref{prop:NLW1D} can be transformed to nonstationary solutions with
singular support on any ray off the light cone.

\section{Special products of distributions}

The subsequent analysis requires refined estimates for products of Sobolev functions whose Fourier transform is supported in $\Gamma = [0,\infty)\subset \R$. For later reference we present results for $\R^n$.

 Let $\Gamma$ be a closed, acute, convex cone in $\R^n$ and $s\in\R$. Notation:
\[
   H_\Gamma^s(\R^n) = \{f\in H^s(\R^n): \supp \widehat{f}\subset \Gamma\} = \cS'_\Gamma(\R^n) \cap H^s(\R^n).
\]
Note that the solutions given in Section\;\ref{sec:stationary} locally belong to $H_\Gamma^s(\R)$ at any fixed time $t$, for $\Gamma = [0,\infty)$.

The product of two members $f\in H_\Gamma^{s_1}(\R^n)$ and $g\in H_\Gamma^{s_2}(\R^n)$ is understood in the sense of the Fourier product.

\begin{proposition}\label{prop:product}
(a) Let $s_1\leq 0$, $s_2\leq \frac{n}{2}$ and $f\in H_\Gamma^{s_1}(\R^n)$, $g\in H_\Gamma^{s_2}(\R^n)$. Then $fg\in H_\Gamma^{\sigma}(\R^n)$ for $\sigma < -\frac{n}{2} + s_1 + s_2$.

(b) Let $s_1\geq 0$, $s_2\in\R$ and $f\in H_\Gamma^{s_1}(\R^n)$, $g\in H_\Gamma^{s_2}(\R^n)$. Then $fg\in  H_\Gamma^{\sigma}(\R^n)$ for $\sigma \leq s_1$, $\sigma < -\frac{n}{2} + s_2$.

In both cases, $\|fg\|_{H^\sigma}\leq \|f\|_{H^{s_1}} \|g\|_{H^{s_2}}$ for some constant $C>0$.
\end{proposition}

\begin{proof}
(1) Assume first that $\Gamma$ is the positive coordinate cone
\[
   \Gamma = \{\xi\in\R^n: \xi_i\geq 0, i = 1,\ldots, n\}.
\]
Write $\int_0^\xi$ for the $n$-dimensional integral $\int_{\xi_1}\ldots \int_{\xi_n}$ etc.
The proof of (a) starts with Minkowski's inequality for integrals and the observation that
\[
   \bone_{[0,\xi]}(\eta) = \bone_{[\eta,\infty)}(\xi)
\]
holds for the characteristic functions of the indicated $n$-dimensional intervals. Thus
\begin{eqnarray*}
\|\widehat{f}\ast\widehat{g}\|_{L^2_\sigma}
& = & \left(\int_0^\infty\Big\vert\int_0^\xi \widehat{f}(\xi - \eta)\widehat{g}(\eta)\dd\eta\Big\vert^2\langle\xi\rangle^{2\sigma}\dd\xi\right)^{1/2}\\
& = &  \left(\int_0^\infty\Big\vert\int_0^\infty \bone_{[0,\xi]}(\eta) \widehat{f}(\xi - \eta)\widehat{g}(\eta)\langle\xi\rangle^{\sigma}\dd\eta\Big\vert^2\dd\xi\right)^{1/2}\\
&\leq & \int_0^\infty\left(\int_0^\infty \bone_{[0,\eta]}(\xi)|\widehat{f}(\xi - \eta)|^2|\widehat{g}(\eta)|^2\langle\xi\rangle^{2\sigma}\dd\xi\right)^{1/2}\dd\eta\\
&=& \int_0^\infty\left(\int_\eta^\infty |\widehat{f}(\xi-\eta)|^2|\widehat{g}(\eta)|^2\langle\xi\rangle^{2\sigma}\dd\xi\right)^{1/2}\dd\eta\\
&=& \int_0^\infty\left(\int_0^\infty |\widehat{f}(\xi)|^2|\widehat{g}(\eta)|^2\langle\xi+\eta\rangle^{2\sigma}\dd\xi\right)^{1/2}\dd\eta.
\end{eqnarray*}
For $\xi\geq 0$, $\eta\geq 0$ and $s_1\leq 0$, $\sigma- s_1\leq 0$ (which holds for the $\sigma$ under consideration provided $s_2\leq\frac{n}{2}$) one has
\[
  \langle\xi+\eta\rangle^{2\sigma} = \langle\xi+\eta\rangle^{2s_1}\langle\xi+\eta\rangle^{2\sigma- 2s_1}
    \leq \langle\xi\rangle^{2s_1}\langle\eta\rangle^{2\sigma- 2s_1}
     =\langle\xi\rangle^{2s_1}\langle\eta\rangle^{2s_2 + 2\sigma- 2s_1- 2s_2}.
\]
Thus
\begin{eqnarray*}
\|\widehat{f}\ast\widehat{g}\|_{L^2_\sigma} &\leq &
    \left(\int_0^\infty |\widehat{f}(\xi)|^2\langle\xi\rangle^{2s_1}\dd\xi\right)^{1/2}
         \int_0^\infty |\widehat{g}(\eta)|\langle\eta\rangle^{s_2}\langle\eta\rangle^{\sigma - s_1 - s_2}\dd\eta\\
&\leq & \|f\|_{H^{s_1}} \|g\|_{H^{s_2}}\left(\int_0^\infty \langle\eta\rangle^{2\sigma- 2s_1- 2s_2}\dd\eta\right)^{1/2}.
\end{eqnarray*}
The latter integral is finite for $\sigma < -\frac{n}{2} + s_1 + s_2$.

In the situation (b), we estimate, using that $s_1\geq 0$ and $\sigma - s_1 \leq 0$,
\[
  \langle\xi+\eta\rangle^{2\sigma} = \langle\xi+\eta\rangle^{2s_1}\langle\xi+\eta\rangle^{2\sigma- 2s_1}
    \leq \langle\xi\rangle^{2s_1}\langle\eta\rangle^{2s_1}\langle\eta\rangle^{2\sigma- 2s_1}
     =\langle\xi\rangle^{2s_1}\langle\eta\rangle^{2s_2 + 2\sigma- 2s_2}.
\]
This time using H\"older's inequality we get
\begin{eqnarray*}
\|\widehat{f}\ast\widehat{g}\|_{L^2_\sigma}
&\leq &
    \left(\int_0^\infty |\widehat{f}(\xi)|^2\langle\xi\rangle^{2s_1}\dd\xi\right)^{1/2}
         \int_0^\infty |\widehat{g}(\eta)|\langle\eta\rangle^{s_2}\langle\eta\rangle^{\sigma - s_2}\dd\eta\\
&\leq & \|f\|_{H^{s_1}} \|g\|_{H^{s_2}}\left(\int_0^\infty \langle\eta\rangle^{2\sigma- 2s_2}\dd\eta\right)^{1/2},
\end{eqnarray*}
which is finite since $\sigma < -\frac{n}{2} + s_2$.

(2) If $\Gamma$ is an arbitrary convex cone, note that if the supports of $\widehat{f}$ and $\widehat{g}$ are contained in $\Gamma$, so is the support of $\widehat{f}\ast\widehat{g}$. Since $\Gamma$ is acute, one may assume (after rotation) that $\Gamma \subset \{(\xi',\xi_n)\in\R^n:\xi_n \geq \alpha|\xi'|\}$ for some $\alpha > 0$; here $\xi' = (\xi_1,\ldots,\xi_{n-1})$. Using a dilation $\widetilde{\xi_n} = \beta\xi_n$, $\widetilde{\xi'} = \xi'$ one may make the opening angle arbitrarily small. After a further rotation, one may assume that
$\Gamma \subset \{\xi\in\R^n: \xi_i\geq 0, i = 1,\ldots, n\}$, that is case (1). Rotations and dilations do not change $H^s(\R^n)$.
\end{proof}
\begin{remark}
The estimates in (a) are sharp. For example, when $n = 1$, $s\leq 0$ and $f\in H_\Gamma^{s}(\R)$, (a) implies that $f^2\in H_\Gamma^{\sigma}(\R)$ for $\sigma < -\frac12 + 2s$. The bounds are realized by $f(x) = (x+\ii 0)^{-1}$, which belongs to $H_\Gamma^{s}(\R)$ if and only if $s < - \frac12$, while $f^2(x) = (x+\ii 0)^{-2}$ belongs to $H_\Gamma^{\sigma}(\R)$ if and only if $\sigma < - \frac32$ as predicted.

The estimates in (b) are not sharp. For example, if $s=s_1=s_2 > \frac{n}2$, $fg$ is known to belong to $H_\Gamma^{s}(\R^n)$, while (b) only predicts $fg\in H_\Gamma^{\sigma}(\R)$ for $\sigma < -\frac{n}2 + s$.

On the other hand, if $f$ and $g$ merely belong to $H^{s_1}(\R^n)$ and $H^{s_2}(\R^n)$, then $fg$ are only known to belong to $H^\sigma(\R^n)$ for $\sigma < -\frac{n}2 + s_1 + s_2$ under the condition that $s_1 + s_2 \geq 0$
\cite[Theorem 8.2.1]{Hoermander:1997}, see also \cite[Lemma 1.3]{Beals:1985} and \cite[formula (1.5)]{Beals:1989}.
\end{remark}

In view of the intended application to the semilinear wave equation \eqref{eq:NLW} we now assume that
\[
   u \in H^s_\Gamma(\R^n)
\]
where $\Gamma$ is a cone as above. Let $p$ be a positive integer. We wish to determine ranges for $\sigma$ such that $u^p \in H^\sigma_\Gamma(\R^n)$.
\begin{remark}\label{rem:IntegerPowers}
(Sobolev properties of integer powers)

\underline{The case $s\leq 0$}. Here
\[
   u^p \in H^\sigma_\Gamma(\R^n) \quad {\rm for}\quad \sigma < -\frac{n}2(p-1) + ps.
\]

\underline{The case $0 < s\leq \frac{n}2$}. Here Proposition \ref{prop:product}(a) immediately gives that
\[
   u^p \in H^\sigma_\Gamma(\R^n) \quad {\rm for}\quad \sigma < -\frac{n}2(p-1) + (p-1)s = (p-1)(s -\frac{n}2).
\]
This follows by induction, using Proposition \ref{prop:product}(a) and (b). Indeed, $u^2 \in H^\sigma_\Gamma(\R^n)$ for $\sigma < -\frac{n}2 + s$ by item (b). For $p = 3$ we use item (a) with $s_1 = -\frac{n}2 + s$, $s_2 = s$ to obtain $u^3 \in H^\sigma_\Gamma(\R^n)$ for $\sigma < -n + 2s$. For $p = 4$ we use again item (a) with $s_1 = -n + 2s$, $s_2 = s$ to obtain $u^4 \in H^\sigma_\Gamma(\R^n)$ for $\sigma < -\frac32 n + 3s$, and so on.

Note that in particular for $s = \frac{n}2$, $u^p \in H^\sigma_\Gamma(\R^n)$ for all $p$ and $\sigma < 0$.

\underline{The case $s > \frac{n}2
$}. Here $u^p \in H^s_\Gamma(\R^n)$ for all $p$ since the latter space is an algebra.

In all cases, $\|u^p\|_{H^\sigma} \leq C\|u\|^p_{H^s}$ for some constant $C>0$.
\end{remark}

The application of Proposition \ref{prop:product} does not produce new estimates for the power function on $H^s_\Gamma(\R^n)$ for $n\geq 2$ and $s\leq \frac{n}2$. We henceforth concentrate on the case $n=1$. In the context of the semilinear wave equation, the question arises whether $u\in H^s_\Gamma(\R)$ implies $u^p\in H^{s-1}_\Gamma(\R)$, that is, whether the $\sigma$ in Remark \ref{rem:IntegerPowers} can attain a value $\geq s-1$. The answer is summarized in the following remark, which is of interest when $s\leq \frac12$.
\begin{remark}\label{rem:stosminus1}
Suppose that $u\in H^s_\Gamma(\R)$, where $\Gamma$ is a closed half-ray. Then $u^p\in H^{s-1}_\Gamma(\R)$ in the following cases:
\begin{equation}\label{eq:rangesPS}
\begin{array}{ll}
p = 2, & -\frac12 < s < \infty,\\[8pt]
p\geq 3, & \frac{1}{2}-\frac{1}{2p-4} < s < \infty.
\end{array}
\end{equation}
In all cases, $\|u^p\|_{H^{s-1}} \leq C\|u\|^p_{H^s}$ for some constant $C>0$.
\end{remark}

\section{Application to 1D-semilinear wave equations}
\label{sec:1DNLW}

In this section, we address wellposedness in $H_\Gamma^s$ of the Cauchy problem for the semilinear wave equation
\begin{equation}\label{eq:1DNLW}
  \p_t^2 u - \p_x^2 u = \pm u^p,\quad u(x,0) = u_0(x), \ \p_t u(x,0) = u_1(x)
\end{equation}
in one space dimension; here $p\geq 2$ is a positive integer. We denote by $\Gamma \subset \R$ a closed half-ray which may be assumed to be the half-line $[0,\infty)$.
\begin{theorem}\label{thm:1DNLW}
Assume that $p$ and $s$ are in the range given by \eqref{eq:rangesPS}. Let $u_0\in H^{s}_\Gamma(\R)$, $u_1\in H^{s-1}_\Gamma(\R)$. Then there is $T > 0$ such that problem \eqref{eq:1DNLW} has a unique distributional solution in
$\cC([-T,T]:H^s_\Gamma(\R))\cap \cC^1([-T,T]:H^{s-1}_\Gamma(\R))$. Further, the map $(u_0, u_1)\to u$ is locally Lipschitz continuous.
\end{theorem}

\begin{proof}
Let $E(t,\cdot) = \cF^{-1}\Big(\frac{\sin t|\xi|}{|\xi|}\Big)$ and
\[
\Big(\cM u\big)(t) = \frac{\dd}{\dd t}E(t)\ast u_0 + E(t)\ast u_1 + \int_0^t E(t-\tau)\ast u^p(\tau)\dd\tau.
\]
The goal is to construct a fixed point $u = \cM u$ in the ball
\[
\cB_T = \{u\in\cC([-T,T]:H^s_\Gamma(\R)): \sup_{-T\leq t \leq T}\|u(t) - \tfrac{\dd}{\dd t}E(t)\ast u_0 - E(t)\ast u_1\|_{H^s(\R)}\leq 1\}
\]
for small $T$. Note that (for $t\geq 0$)
\[
\widehat{\cM u}(\xi,t) = \cos (t|\xi|)\widehat{u_0}(\xi) + \frac{\sin t|\xi|}{|\xi|}\widehat{u_1}(\xi) + \int_0^t \frac{\sin (t-\tau)|\xi|}{|\xi|}\widehat{u}^{\ast p}(\xi,\tau)\dd\tau
\]
from where
\[
\|\cM u(t)\|_{H^s(\R)} \leq \|u_0\|_{H^s(\R)} + \|u_1\|_{H^{s-1}(\R)} + C\int_0^t \|u^p(\tau)\|_{H^{s-1}(\R)}\dd\tau.
\]
The case $p = 2$, $s > -1/2$. Let $v \in H^s_\Gamma(\R)$. Then $v^2 \in H^\sigma_\Gamma(\R)$ for $\sigma < - \frac12 + 2s$, say $\sigma = - \frac12 + 2s - \eps$.

We have $- \frac12 + 2s - \eps > s - 1$ for small $\eps > 0$ and so
\[
   v^2 \in H^{s-1}_\Gamma(\R),\qquad \|v^2\|_{H^{s-1}(\R)} \leq C \|v\|^2_{H^s(\R)}.
\]
Similarly, if $v, w \in H^\sigma_\Gamma(\R)$ then
\[
   v^2 - w^2 \in H^{s-1}_\Gamma(\R),\qquad \|v^2 - w^2\|_{H^{s-1}(\R)} \leq C \|v-w\|_{H^s(\R)}\|v+w\|_{H^s(\R)}.
\]
It follows that
\begin{itemize}
\item[(a)] $u \in \cC([-T,T]:H^s_\Gamma(\R)) \Rightarrow u^2 \in \cC([-T,T]:H^{s-1}_\Gamma(\R))$. Indeed, apply the estimate above to $v = u(t+h)$, $w = u(t)$.
\item[(b)] $\cM:\cB_T \to \cB_T$. This follows from the estimate
\[
\|\cM u(t)\|_{H^s(\R)} \leq \|u_0\|_{H^s(\R)} + \|u_1\|_{H^{s-1}(\R)} + C\int_0^t \|u(\tau)\|^p_{H^{s}(\R)}\dd\tau.
\]
\item[(c)] $\cM$ is a contraction on $\cB_T$ for small $T$. This follows from the similar estimate
\[
\|\cM u(t) - \cM v(t)\|_{H^s(\R)} \leq  C\int_0^t \|u(\tau) - v(\tau)\|_{H^{s}(\R)}\sup_{0\leq \tau\leq T}\|u(\tau) - v(\tau)\|_{H^{s}(\R)}\dd\tau.
\]
\end{itemize}
Existence and uniqueness follow. Also, $u = \cM u$ and $u \in \cC([-T,T]:H^s_\Gamma(\R))$ implies that
$u \in \cC^1([-T,T]:H^{s-1}_\Gamma(\R))$. Finally, Gronwall's inequality gives local Lipschitz continuity.

For $p\geq 3$ a similar argument works using factorization of $v^p - w^p$ and Proposition \ref{prop:product}.
\end{proof}

\begin{remark}
Note that the lower bound in \eqref{eq:rangesPS}  is smaller than $s_{\rm sob}$ in \eqref{eq:sob} only for $p=2$ and $p = 3$, so in these cases Theorem \ref{thm:1DNLW} improves the results of \cite{Christ:2003,Forlano:2020}.
\end{remark}

\section{Extensions to semilinear wave equations in higher dimensions}\label{sec:ndim}

The results of the previous sections focused on new phenomena for semilinear wave equations on the real line. We now discuss their extension to higher dimensions.
We shall again give explicit solutions to \eqref{eq:NLW} exhibiting anomalous propagation of singularities. In these examples, the nonlinear term $u^p$ even exists in the classical sense of the $p$-th power of an $L^p_{\rm loc}$-function, which in this case coincides with the $p$-th power taken in the sense of the Fourier product.

In addition to $(x+\ii 0)^{\lambda}$, which was considered above, we introduce the radially symmetric pseudofunction $r^\lambda\in\cS'(\R^n)$, defined by
\[
   \langle r^\lambda,\varphi\rangle = \int|x|^\lambda \varphi(x)\dd x
\]
for $\Re \lambda > -n$. It can be extended to a meromorphic $\cS'(\R^n)$-valued function of $\lambda \in \C$ with simple poles at $\lambda = - n -2k, k\in \N$, \cite[Section I.3.9]{GelfandShilov:1964}.
Its Fourier transform is given by
\[
\cF\big(r^\lambda\big)(\rho) = \pi^{-\lambda - n/2}\frac{\Gamma\left(\frac{\lambda + n}{2}\right)}{\Gamma\left(\frac{-\lambda}{2}\right)}\rho^{-\lambda - n}
\]
for $\lambda\neq -n-2k$ and $\lambda\neq 2k, k\in\N$  \cite[Section II.3.3]{GelfandShilov:1964}, \cite[Formula (VII,7;13)]{Schwartz:1966}.

\begin{remark}
The product of the pseudofunctions $r^\lambda$ is defined as a Fourier product: First, one may extend the definition to $\lambda \in\C$ by taking finite parts in the poles. It was proved in \cite[Satz 5]{Ortner:1980} that the $\cS'$-convolution of ${\rm Pf\,} r^\alpha$ and ${\rm Pf\,} r^\beta$ exists if and only if $\Re (\alpha + \beta) < -n$. This can be used to characterize the range of exponents for which the Fourier product exists. However, for the present paper, only the range $\lambda\in \R$, $-n < \lambda < 0$ will be needed, in which case both $r^\lambda$ and $\cF(r^\lambda)$ are locally integrable functions. We show that -- in the indicated range of exponents -- the Fourier product of $r^\lambda$ and $r^\mu$ exists if $\lambda + \mu > -n$. Up to constant factors, the respective Fourier transforms are $\rho^{-\lambda-n}$ and $\rho^{-\mu-n}$. Take $\varphi \in \cS(\R^n)$. Then $(\varphi \ast {\rho}^{-\lambda-n})\rho^{-\mu-n} \in \cE'(\R^n) + L^1(\R^n) \subset \cD_{L^1}'(\R^n)$ provided $-\lambda - n -\mu -n + n - 1 < -1$. This is exactly the case when $\lambda +\mu > -n$. Thus the Fourier product of $r^\lambda$ and $r^\mu$ exists in this range. A proof that $r^\lambda\!\cdot\! r^\mu = r^{\lambda+\mu}$ can be found, e.g., in \cite[Example 5.4]{MO:1992}.
Incidentally, $r^\lambda$, $r^\mu$ and $r^{\lambda+\mu}$ are locally integrable functions in the range $ 0 > \lambda, \mu > -n$, $\lambda +\mu > -n$, and the usual product equals $r^\lambda r^\mu =r^{\lambda+\mu}$, thus coincides with the Fourier product. The same holds for integer powers $(r^\lambda)^p = r^{\lambda p}$ when $\lambda p > -n$.
\end{remark}

Outside the poles, the pseudofunctions $r^\lambda$ satisfy
\[
   \Delta r^\lambda = \lambda(\lambda + n -2)r^{\lambda - 2}.
\]
In particular, when $\lambda > 2-n$ and $p = 1 - 2/{\lambda}$, $r^\lambda$ belongs to $L^p_{\rm loc}(\R^n)$, $(r^\lambda)^p = r^{\lambda p}$ and it satisfies the elliptic equation
\[
   \Delta r^\lambda = \lambda(\lambda+n-2)(r^\lambda)^p,
\]
where the derivatives are understood in the weak sense and the $p$th power as the evaluation of the Nemytskii operator $L^p_{\rm loc}(\R^n) \to L^1_{\rm loc}(\R^n)$.

\begin{remark} \label{rem:HsPrpoertiesOfPseudofunctions2}
The following properties are easy to show. We assume here that $\lambda < 0$ so that the pseudofunction $\rho^{-\lambda - n}$ is locally integrable.

\noindent (a) For $\lambda < 0$, the following equivalences hold:
\begin{itemize}
\item[(1)]  $r^\lambda \in H^s_{\rm loc}(\R^n)\ \Leftrightarrow\ s < \lambda + \frac{n}2$,
\item[(2)]  $r^\lambda \in H^s(\R^n)\ \Leftrightarrow\ \lambda < -\frac{n}2 \ {\rm and\ }s < \lambda + \frac{n}2$.
\end{itemize}
(b) If $-n < \lambda < 0$, both $r^\lambda$ and its Fourier transform belong to $L^1_{\rm loc}(\R^n)$ and
\[
r^\lambda \in H^s_{\rm loc}(\R^n) \Leftrightarrow s < \lambda + \frac{n}{2}.
\]
The conditions $\lambda - 2 > - n$ and $\lambda < 0$ can only be satisfied if $n\geq 3$.
In the ranges of $\lambda$ under consideration, $r^\lambda$ belongs to $L^p_{\rm loc}(\R^n)$.
\end{remark}
\begin{proposition}\label{prop:NLWnD}
Let $n\geq 3$. For every $s<\frac{n}2$ there are $\lambda<0$ and $p\in\N$ such that

(a) the distribution $u_0(x) = r^\lambda$ belongs to $H^s_{\rm loc}(\R)$, $\singsupp u_0 = \{0\}$, and

(b) $u(x,t)\equiv u_0(x)$ is a distributional solution to the semilinear wave equation
\begin{equation}\label{eq:NLWnD}
  \p_t^2 u - \Delta u = -\lambda(\lambda -1)u^p,\quad u(x,0) = u_0(x), \ \p_t u(x,0) = 0,
\end{equation}
where the nonlinear term is understood in the sense of the Nemytskii operator $L^p_{\rm loc}(\R^n) \to L^1_{\rm loc}(\R^n)$. Its singular support is the noncharacteristic line $\{(0,t):t\in\R\}$.
\end{proposition}
\begin{proof}
It is clear that the function $u(x,t)$ satisfies the semilinear wave equation \eqref{eq:NLWnD} when $\lambda = \frac2{1-p}$.
As noted, $r^{\lambda}$ belongs to $H^s_{\rm loc}(\R)$ iff $s < \lambda + \frac{n}2 = \frac2{1-p} + \frac{n}2$. Let $p\to\infty$ to produce the desired $s$.
\end{proof}
\begin{remark} (a) When $s>\frac{n+1}2$, the anomalous singular support of the solution from Proposition \ref{prop:NLWnD} would be ruled out by the results in \cite{Rauch:1979}.
Thus there is a gap between the counterexamples in Proposition \ref{prop:NLWnD} ($s < \frac{n}2$) and the nonlinear propagation results for $s>\frac{n+1}2$.

(b) In order to have that $r^\lambda \in H^s(\R^n)$ it is necessary that $\lambda < -\frac{n}2$ and $s < \lambda + \frac{n}2$ (Remark \ref{rem:HsPrpoertiesOfPseudofunctions2}). This cannot occur for $p\geq 2$ and $n\geq 2$.
\end{remark}

It remains to be checked whether other radial solutions to the semilinear Laplace equation \cite{Cote:2018,Quittner:2019} can serve for constructing anomalous solutions to \eqref{eq:NLW}.

Lorentz transformations can be applied, similar to the one-dimensional problem, to transform the stationary solutions
$u_0(x) = r^\lambda$ from Proposition \ref{prop:NLWnD}
to time dependent solutions with singular support on noncharacteristic rays.

 Indeed, suppose that $u(x,t)\equiv u_0(x)$ is a stationary solution to
\[
  \p_t^2 u - \Delta u = f(u).
\]
Let $|c| < 1$ and set
\[
   v(x_1,\ldots,x_n,t) = u_0\Big(\frac{x_1+ct}{\sqrt{1-c^2}}, x_2, \ldots, x_n\Big).
\]
Then
\[
  \p_t^2 v - \Delta v = - f(v).
\]
The singular support of $v$ is the ray $\{x_1 + ct = 0, x_2 = 0,\ldots, x_n=0: t \in \R\}$, which lies inside the light cone.

Similarly, the choice of $|c|> 1$ and
\[
   w(x_1,\ldots,x_n,t) = u_0\Big(\frac{\pm x_1+ct}{\sqrt{c^2-1}}, x_2, \ldots, x_n\Big)
\]
gives a solution to $\p_t^2 w - \Delta w = f(w)$ with singular support on a ray in a coordinate plane outside the light cone. By means of a rotation of the $x$-coordinate system, one may produce solutions whose singular support is any ray not contained in the light cone.

\section*{Acknowledgements}

We thank Justin Forlano, Oana Pocovnicu and Jeffrey Rauch for fruitful discussions.


\begin{thebibliography}{00}



\bibitem{Beals:1983}
Michael Beals.
\newblock Self-spreading and strength of singularities for solutions to
  semilinear wave equations.
\newblock {\em Ann. of Math. (2)}, 118(1):187--214, 1983.


\bibitem{Beals:1985}
Michael Beals.
\newblock Propagation of smoothness for semilinear second-order strictly hyperbolic differential equations with data singular at one point. In: F. Tr\`eves (Ed.), Pseudodifferential operators and applications.
Proceedings of Symposia in Pure mathematics, vol. 43. American Mathematical Society, Providence RI, 1985, 21--44.

\bibitem{Beals:1989}
Michael Beals.
\newblock {\em Propagation and interaction of singularities in nonlinear
  hyperbolic problems}, volume~3 of {\em Progress in Nonlinear Differential
  Equations and their Applications}.
\newblock Birkh\"{a}user Boston, Inc., Boston, MA, 1989.

\bibitem{BealsReed:1982}
Michael Beals and Michael Reed. Propagation of singularities for hyperbolic pseudodifferential operators
with nonsmooth coefficients. {\em Comm. Pure Appl. Math.} 35 (1982), 169-184.

\bibitem{Bizon:2007}
Piotr Bizo\'{n}, Dieter Maison, and Arthur Wasserman.
\newblock Self-similar solutions of semilinear wave equations with a focusing
  nonlinearity.
\newblock {\em Nonlinearity}, 20(9):2061--2074, 2007.

\bibitem{Bony:1981}
Jean-Michel Bony. Calcul symbolique et propagation des singularit\'es pour les \'equations aux d\'eriv\'ees
partielles nonlin\'eaires. {\em Ann. Scient. \'Ec. Norm. Sup.}, s\'erie 4, 14 (1981), 209--246.

\bibitem{Burq:2008}
Nicolas Burq and Nikolay Tzvetkov.
\newblock Random data {C}auchy theory for supercritical wave equations. {I}.
  {L}ocal theory.
\newblock {\em Invent. Math.}, 173(3):449--475, 2008.



\bibitem{Christ:2003}
Michael Christ, James Colliander and  Terence Tao.
Ill-posedness for nonlinear Schr\"odinger and wave equations.
arXiv:math/0311048, 2003.

\bibitem{Cote:2018}
Rapha\"el C\^{o}te, Carlos~E. Kenig, Andrew Lawrie, and Wilhelm Schlag.
\newblock Profiles for the radial focusing {$4d$} energy-critical wave
  equation.
\newblock {\em Comm. Math. Phys.}, 357(3):943--1008, 2018.

%

\bibitem{DeLellis:1999}
Camillo De Lellis and L\'{a}szl\'{o} Sz\'{e}kelyhidi, Jr. On turbulence and geometry: from Nash to Onsager.
{\em Notices Amer. Math. Soc.}, 66(5):677--685, 2019.

\bibitem{Forlano:2020}
Justin Forlano and Mamoru Okamoto.
A remark on norm inflation for nonlinear wave equations.
{\em Dynamics of Partial Differential Equations}, 17(4):361--381, 2020.

\bibitem{GelfandShilov:1964}
Izrail' M. Gel'fand and Georgi E. Shilov.
\newblock Generalized functions, volume I: {P}roperties and operations.
Academic Press, New York-London, 1964.

%
%


\bibitem{GerardRauch:1987}
Patrick G\'erard and Jeffrey Rauch.
Propagation de la r\'egularit\'e locale de solutions d’\'equations hyperboliques non lin\'eaires.
{\em Annales de l’Institut Fourier} 37(3)(1987), 65--84.


\bibitem{Hoermander:1971}
Lars H\"{o}rmander.
\newblock Fourier integral operators. {I}.
\newblock {\em Acta Math.}, 127(1-2):79--183, 1971.

\bibitem{Hoermander:1997}
Lars H\"{o}rmander.
\newblock {\em Lectures on nonlinear hyperbolic differential equations},
  volume~26 of {\em Math\'{e}matiques \& Applications (Berlin)}.
\newblock Springer-Verlag, Berlin, 1997.


\bibitem{Kapitanski:1994}
Lev Kapitanski. Weak and yet weaker solutions of semilinear wave equations. {\em Comm. Part. Diff. Eq.} 19
(1994), 1629--1676.

\bibitem{Kato:2007}
Jun Kato, Makoto Nakamura, and Tohru Ozawa.
\newblock A generalization of the weighted {S}trichartz estimates for wave
  equations and an application to self-similar solutions.
\newblock {\em Comm. Pure Appl. Math.}, 60(2):164--186, 2007.

\bibitem{Keel:1998}
Markus Keel and Terence Tao. Endpoint Strichartz estimates. {\em Amer. J. Math.} 120 (1998), no. 5, 955–980.

\bibitem{Lindblad:1993}
Hans Lindblad. A sharp counterexample on the local existence of low-regularity solutions to nonlinear wave
equations. {\em Duke Math.} J. 72 (1993), no. 2, 503–539.

\bibitem{Lindblad:1996}
Hans Lindblad.
\newblock Counterexamples to local existence for semi-linear wave equations.
\newblock {\em Amer. J. Math.}, 118(1):1--16, 1996.

\bibitem{LindbladSogge:1995}
Hans Lindblad and Christopher D. Sogge. On existence and scattering with minimal regularity for semilinear wave
equations. {\em Jour. Func. Anal.} 130 (1995), 357--426.

\bibitem{MO:1992}
Michael Oberguggenberger.
\newblock {\em Multiplication of distributions and applications to partial
  differential equations}, volume 259 of {\em Pitman Research Notes in
  Mathematics Series}.
\newblock Longman Scientific \& Technical, Harlow, 1992.

\bibitem{MO:2021}
Michael Oberguggenberger.
Anomalous solutions to nonlinear hyperbolic equations. In: M. Cicognani, D. Del Santo, A. Parmeggiani, M. Reissig (Eds.), Anomalies in Partial Differential Equations, Springer INdAM Series, Cham 2021, 347 - 367.

\bibitem{Oh:2016}
Tadahiro Oh and Oana Pocovnicu.
\newblock Probabilistic global well-posedness of the energy-critical defocusing
  quintic nonlinear wave equation on {$\mathbb{R}^3$}.
\newblock {\em J. Math. Pures Appl. (9)}, 105(3):342--366, 2016.

\bibitem{Oh:2021}
Tadahiro Oh, Tristan Robert, Philippe Sosoe, and Yuzhao Wang.
\newblock On the two-dimensional hyperbolic stochastic sine-{G}ordon equation.
\newblock {\em Stoch. Partial Differ. Equ. Anal. Comput.}, 9(1):1--32, 2021.

\bibitem{Oh:2022}
Tadahiro Oh, Oana Pocovnicu, and Nikolay Tzvetkov.
\newblock Probabilistic local {C}auchy theory of the cubic nonlinear wave
  equation in negative {S}obolev spaces.
\newblock {\em Ann. Inst. Fourier (Grenoble)}, 72(2):771--830, 2022.

\bibitem{Ortner:1980}
Norbert Ortner. Faltung hypersingul\"{a}rer Integraloperatoren. {\em Math. Ann.}, 248(1):19--46, 1980.

\bibitem{Ortner:2023}
Norbert Ortner and Peter Wagner.
A distributional version of Frullani's integral.
\newblock {\em Bull. Sci. Math.}, 186: Paper No. 103272, 2023.

\bibitem{Pecher:2000a}
Hartmut Pecher.
\newblock Self-similar and asymptotically self-similar solutions of nonlinear
  wave equations.
\newblock {\em Math. Ann.}, 316(2):259--281, 2000.

\bibitem{Pecher:2000}
Hartmut Pecher.
\newblock Sharp existence results for self-similar solutions of semilinear wave
  equations.
\newblock {\em NoDEA Nonlinear Differential Equations Appl.}, 7(3):323--341,
  2000.

%

\bibitem{Quittner:2019}
Pavol Quittner and Philippe Souplet. Superlinear Parabolic Problems. Blow-up, Global Existence and Steady States.
2nd Ed., Birkh\"auser Advanced Texts Basler Lehrb\"ucher, Springer, Cham 2019.


\bibitem{Rauch:1979}
Jeffrey Rauch.
\newblock Singularities of solutions to semilinear wave equations.
\newblock {\em J. Math. Pures Appl. (9)}, 58(3):299--308, 1979.

\bibitem{Rauch:1981}
Jeffrey Rauch. I. The $u^5$-Klein-Gordon equation; II. Anomalous singularities for semilinear wave equations. In: H. Brezis, J.L. Lions (Eds.),
Nonlinear partial differential equations and their applications. Coll\`ege de France Seminar, Vol. I (Paris, 1978/1979)
Pitman Research Notes in Mathematics Vol. 53, Pitman, Boston 1981, 335--364.

\bibitem{RauchReed:1980}
Jeffrey Rauch and Michael Reed.
\newblock Propagation of singularities for semilinear hyperbolic equations in
  one space variable.
\newblock {\em Ann. of Math. (2)}, 111(3):531--552, 1980.

\bibitem{RauchReed:1982}
Jeffrey Rauch and Michael Reed.
\newblock Nonlinear microlocal analysis of
semilinear hyperbolic systems in one
space dimension.
\newblock {\em Duke Math. J.}, 49(2):397--475, 1982.


\bibitem{RauchReed:1981}
Jeffrey Rauch and Michael Reed.
\newblock Jump discontinuities of semilinear, strictly hyperbolic systems in
  two variables: creation and propagation.
\newblock {\em Comm. Math. Phys.}, 81(2):203--227, 1981.

\bibitem{Reed:1978}
Michael Reed.
\newblock Propagation of singularities for non-linear wave equations in one
  dimension.
\newblock {\em Comm. Partial Differential Equations}, 3(2):153--199, 1978.

\bibitem{Ribaud:2002}
Francis Ribaud and Abdellah Youssfi.
\newblock Global solutions and self-similar solutions of semilinear wave
  equation.
\newblock {\em Math. Z.}, 239(2):231--262, 2002.

\bibitem{Schwartz:1966}
Laurent Schwartz.
\newblock {\em Th\'{e}orie des distributions}.
\newblock Publications de l'Institut de Math\'{e}matique de l'Universit\'{e} de
  Strasbourg, No. IX-X. Nouvelle \'{e}dition, enti\'{e}rement corrig\'{e}e,
  refondue et augment\'{e}e. Hermann, Paris, 1966.

\bibitem{Tao:1999}
Terence Tao. Low regularity semi-linear wave equations. {\em Comm. Partial Differential Equations}, 24 (1999),
no. 3-4, 599--629.


\bibitem{Tzvetkov:2019}
Nikolay Tzvetkov.  Random data wave equations, Singular random dynamics. In: Franco Flandoli, Massimiliano Gubinelli, Martin Hairer (Eds.),
Singular Random Dynamics. Lecture Notes in Mathematics, vol. 2253,  Springer, Cham 2019, 221--313.


\bibitem{Vladimirov:1981}
 Vasilij S. Vladimirov. Le distribuzioni nella fisica matematica. Edizioni Mir, Mosca 1981.

\end{thebibliography}

\end{document}